\documentclass[12pt, emspublic, reqno]{amsart}

\usepackage{amsfonts}
\usepackage{amssymb, euscript, mathrsfs}

\newtheorem{theorem}{Theorem}[section]
\newtheorem{lemma}[theorem]{Lemma}
\newtheorem{proposition}[theorem]{Proposition}

\numberwithin{equation}{section}

\def\Cs{\mathscr{C c 1234}}

\def\Cal{\mathcal}

\def\C{{\Cal C}}

\def\S{{\Cal S}}

\def\I{{\Cal I}}

\def\f0{f_0}
\def\Fc0{\varphi_0}
\def\rn{\bbr^n}

\def\I_k {I_{-}^{k/2}}
\def\I+k {I_{+}^{k/2}}

\def\cd{\stackrel{*}{\C}\!{}_{m, k}^\a}
\def\sd{\stackrel{*}{\S}\!{}_{m, k}^\a}
\def\cd0{\stackrel{*}{\C}\!{}_{m, k}^\a}
\def\sd0{\stackrel{*}{\S}\!{}_{m, k}^\a}

\def\ncd0{\stackrel{*}{\Cs}\!{}_{m, k}^\a}

\def\bbr{{\Bbb R}}

\def\bbn{{\Bbb N}}

\def\bbc{{\Bbb C}}

\def\min{{\hbox{\rm min}}}

\def\part{\partial}

\def\Gam{\Gamma}

\def\a{\alpha}

\def\Del{\Delta}
\def\del{\delta}
\def\vp{\varphi}

\def\gam{\gamma}

\def\sig{\sigma}
\def\lam{\lambda}

\def\e{\varepsilon}

\def\min{{\hbox{\rm min}}}

\def\part{\partial}

\def\Gam{\Gamma}

\def\a{\alpha}

\def\sideremark#1{\ifvmode\leavevmode\fi\vadjust{\vbox to0pt{\vss
 \hbox to 0pt{\hskip\hsize\hskip1em
\vbox{\hsize2cm\tiny\raggedright\pretolerance10000
 \noindent #1\hfill}\hss}\vbox to8pt{\vfil}\vss}}}%

\newcommand{\be}{\begin{equation}}
\newcommand{\ee}{\end{equation}}
\newcommand{\bea}{\begin{eqnarray}}
\newcommand{\eea}{\end{eqnarray}}
\newcommand{\Bea}{\begin{eqnarray*}}
\newcommand{\Eea}{\end{eqnarray*}}
\begin{document}

\title[Elementary  Inversion of Riesz Potentials]
{Elementary  Inversion of Riesz Potentials and Radon-John Transforms}

\author{  B. Rubin }
\address{Department of Mathematics, Louisiana State University, Baton Rouge,
Louisiana 70803, USA}
\email{borisr@math.lsu.edu}

\thanks{ The work was
 supported  by the NSF grants PFUND-137 (Louisiana Board of Regents),
 DMS-0556157, and the Hebrew University of Jerusalem.}

\subjclass[2000]{Primary 44A12; Secondary 47G10}



\keywords{Riesz potentials, Radon transforms, Wavelet transforms.}
\begin{abstract}
New simple proofs are given to  some elementary  approximate  and explicit inversion formulae for  Riesz potentials.  The results are applied to   reconstruction of functions from their integrals over
 Euclidean planes in integral geometry.
\end{abstract}

\maketitle

\section{Introduction}\label {hujkz}

The  Riesz potential of
order $\a$ of a sufficiently good function $f$ on $\bbr^n$ is
defined by
\be\label{rpot} (I^\a f)(x)=\frac{1}{\gamma_n(\a)}
\int_{\bbr^n} \frac{f(y)\,dy}{|x-y|^{n-\a}},\quad
  \gamma_n(\a)=  \frac{2^\a\pi^{n/2}\Gamma(\a/2)}{\Gamma((n-\a)/2)},
  \ee
$$Re \, \a >0, \qquad \; \a \neq n, n+2, n+4, \ldots \,.$$
  This  operator can be regarded (in a
certain sense) as a negative power of ``minus-Laplacian", namely,
\be\label{frp} I^\a =(- \Del)^{-\a/2},\qquad
\Del=\frac{\partial^2}{\partial x_1^2}+ \frac{\partial^2}{\partial
x_2^2}+\ldots + \frac{\partial^2}{\partial x_n^2}.\ee
 If $f\in L^p (\bbr^n)$, then the integral (\ref{rpot}) converges a.e. provided that $1\le p<n/Re \, \a$. This  condition is sharp. Explicit and approximate inversion formulae for $I^\a$ are of great importance. We refer to  \cite{He, La, Ru96, Ru04b, Sa02, SKM, St} for the basic properties and applications of Riesz potentials. In this article we restrict  to real $\a \in (0,n)$. This case reflects basic features and is sufficient for  integral geometric applications in Section \ref{333}.

Our main working tools are a suitably chosen auxiliary function $w(x)$ (we call it a {\it reconstructing function}) and its scaled version $w_t(x)=t^{-n}w(x/t), \; t>0$. In the modern literature such a function is called a {\it wavelet} and a convolution $ (f* w_t)(x)$ is called the {\it continuous wavelet transform} of $f$; cf. \cite{Dau, FJW}. We choose
\be\label {pp6g}
w(x)= (-\Del)^m [(1+|x|^2)^{m-(n+\a)/2}], \qquad \tilde w(x)= (-\Del w)(x),\ee
   $m \in \bbn$, and keep this notation throughout the paper.

  The main results of the article are presented by Theorem A below and Theorem B in Section \ref{333}. The latter gives an example of application of Theorem A
 to elementary inversion of the $k$-plane Radon-John transform in integral geometry and falls into the general scope of the convolution-backprojection method (cf. \cite{Na, Ru04a}). Radon transforms and their generalizations are the basic tools in tomography and numerous related areas of pure and applied mathematics; see \cite {Ag, GGG, Go, He, Ka1, Ka2,  Ka3, Ku, Mar, Na,  RK, Q06} and references therein.
 
\noindent {\bf Theorem A.} {\it  Let $2m>\a$, $f\in L^p (\bbr^n)$, $1\le p<n/\a$. Then
\bea \label {8rfh} c_{\a,m}\, f&=&\lim\limits_{t \to 0} \,(I^\a f  \ast t^{-\a} w_t), \qquad c_{\a,m}=\gamma_n(2m-\a),\\
  \label {t8rfh} d_{\a,m}\, f&=&\int_0^\infty \frac{I^\a f  \ast  \tilde w_t}{t^{1+\a}}\, dt, \qquad d_{\a,m}=
 (2m\!-\!\a)\,c_{\a,m},\eea
 where $\int_0^\infty (\ldots) =\lim\limits_{\e \to 0}\int_\e^\infty (\ldots)$ and $\gamma_n(\cdot)$ has the same meaning as in (\ref{rpot}).
The limit in both formulae exists in the $L^p$-norm and in the a.e. sense. If, moreover, $f\in C_0(\bbr^n)$, then it exists in the $\sup$-norm.}

We recall that $C_0(\bbr^n)$ denotes the space of continuous functions on $\bbr^n$ vanishing at infinity.

Some comments are in order. Statements of this kind are not new. Regarding (\ref{8rfh}), we observe that approximate inversion  of operators of  the potential type was initiated by Zavolzhenskii,  Nogin and Samko \cite {ZN, NS, Sa98, Sa99, Sa02}. Conceptually this approach is close to the convolution-backprojection method in tomography \cite {Na}. An elegant formula (\ref{8rfh})  is due to Samko \cite[p. 325]{Sa02}, \cite{Sa98, Sa99}; see also \cite {RS}. Below we give a new simple proof of it.

Regarding (\ref{t8rfh}),  wavelet-like inversion formulae for Riesz potentials are also well-known; see, e.g., \cite[Section 17]{Ru96}. Our aim is to show that elementary reconstructing functions (\ref{pp6g}) work perfectly in the wavelet inversion scheme and the relevant  justification is much  simpler  than  in the general theory; cf.  \cite{Ru96}. The case $\a=0$ in (\ref{t8rfh}), when $I^\a$ is substituted by the identity operator, represents a variant of Calder\'on's reproducing formula \cite{Dau, FJW}.

The validity of formulae (\ref{8rfh}) and  (\ref{t8rfh}) can be formally
explained in the language of the Fourier transform. We have
 \be \lim\limits_{t \to 0} \,(I^\a f  \ast t^{-\a} w_t)^\wedge (\xi)=\hat f (\xi)\,  \lim\limits_{t \to 0} \,(t|\xi|)^{-\a} \hat w (t\xi)=c_w\,\hat f (\xi)\ee
 provided that the limit $c_w= \lim\limits_{\xi\to 0} \,|\xi|^{-\a} \hat w (\xi)$ exists. Moreover, since $\tilde w$ in (\ref{pp6g}) is a radial function, we easily get
  \be\label{psss} \Big[ \int_0^\infty
\frac {I^\a f \ast \tilde w_t}{t^{1+\a}} \,dt \bigg]^\wedge (\xi) = d_w
\hat f (\xi),\qquad d_w =\frac{1}{\sig_{n-1}}\int_{\bbr^n} \frac {\hat w
(y)\, dy}{|y|^{n+ \a-2}},\ee
where  $\sig_{n-1}\!= \! 2\pi^{n/2} \big/ \Gamma (n/2)$ is the surface area of the unit sphere $|x|\!=\!1$.

This simple argument does not work for arbitrary $L^p$-functions. To cover this case, we develop another technique and give the final answer without using the Fourier transform.

\section{Preliminaries} We recall some known facts. The Fourier transform of a function
$f \in L^1 = L^1 (\bbr^n)$ is defined by \be \label{ft}\hat f (\xi) = \int_{\bbr^{ n}} f(x) e^{ i x \cdot \xi} \,dx, \qquad x \cdot \xi = x_1 \xi_1 + \ldots + x_n\xi_n.
\ee
Let $S=S (\bbr^n)$ be the Schwartz space  of all $C^\infty$ functions which, together with derivatives of all orders, vanish at infinity faster than any inverse power of $|x|=(x_1^2 +\ldots +x_n^2)^{1/2}$. We endow $S(\bbr^n)$ with a standard topology  generated by the sequence of norms $$
||\vp||_m=\max\limits_x(1+|x|)^m \sum_{|j|\le m} |(\part^j\vp)(x)|,
\quad m=0,1,2, \ldots \,.$$

The following spaces adapted to Riesz potentials were
introduced by  Semyanistyi \cite{Se}; further generalizations due to  Lizorkin  and  Samko  can be found in  \cite{Liz, Sa02}.

Let  $\Psi=\Psi(\bbr^n)$ be the subspace of $S$, consisting of functions $\psi (\xi)$ vanishing  at  $\xi=0$ with all  derivatives, and let $\Phi=\Phi(\bbr^n)$ be the Fourier image of $\Psi$. We
equip  $\Phi$ with the induced topology of  $S$.
Then $\Phi$ becomes a  topological vector space which is isomorphic
to $\Psi$ under the action of the Fourier transform. We denote by $\Phi'$ the space of distributions over $\Phi$.

The main reason for introducing the spaces $\Psi$ and $\Phi$ is that $\Psi$ is invariant under multiplication by $|\xi|^\a$ for any $\a\in \bbc$ and therefore, the Riesz potential $I^\a$ is an automorphism
of $\Phi$ and $\Phi'$.
\begin{proposition}\label{c3.9} \cite[p. 21 ]{Ru96} Let $f\in L^r$, $1\le r<\infty$,
and $g\in L^p$, $1\le p<\infty$. If $f=g$ in the $\Phi'$-sense, then
$f=g$ almost everywhere.
 \end{proposition}

\section{Proof of Theorem A}
The following  auxiliary propositions form the heart of the proof.

\begin{lemma} \label {00jy} Let  $w\!\in\! L^1(\bbr^n)$ be such that  $h=I^\a w$ has an integrable decreasing radial majorant.
If $f\!\in \!L^p(\bbr^n)$, $ 1\!\le \!p\!<\!n/\a$,
then \be \label{09k} \lim\limits_{t \to 0} \,(I^\a f  \ast t^{-\a} w_t)=c_w (\a)\, f, \qquad c_w (\a)=\int_{\bbr^n} h(x)\, dx,\ee
where the limit exists in the $L^p$-norm and in the a.e. sense. If, moreover, $f\in C_0(\bbr^n)$, then the limit exists in the $\sup$-norm.
\end{lemma}
\begin{proof}  We have $t^{-\a} (I^\a w_t)(x)=t^{-n}h(x/t)=h_t(x)$.  Then
$$
I^\a f  \ast t^{-\a} w_t=f\ast t^{-\a} I^\a w_t=f\ast h_t,
$$
which yields an approximate identity \cite [p. 62] {St}. The above application of Fubini's theorem is permissible because $(I^\a |f|)  \ast |w|<\infty$ a.e. \end{proof}

There are many ways to choose a  reconstructing function $w$. Usually  $w$ or $I^\a w$ or both are expressed analytically in a pretty complicated way. As we shall see below, an advantage of the choice (\ref {pp6g})  is  that  {\it both} $w$ and $h=I^\a w$ are  elementary  and a constant $c_w (\a)$ can be different from zero.
\begin {lemma}\label {761s}  Let
$w(x)= (-\Del)^m [(1+|x|^2)^{m-(n+\a)/2}]$.  If $\a> 0$, then $w\!\in\! L^1(\bbr^n)$. If, moreover,
$\a <\min (n, 2m)$ and $h=I^\a w$, then
\be\label {78f}
h(x)\!=\!a_{\a,m}\, (1\!+\!|x|^2)^{(\a-n)/2-m}, \quad a_{\a,m}\!=\!2^{2m -\a}\, \frac{\Gam ((n\!+\!2m\! -\!\a)/2)}{\Gam ((n\!+\!\a\!-\!2m)/2)},\nonumber\ee
so that $h\in  L^1(\bbr^n)$ and
\be c_{\a,m}\equiv \int_{\bbr^n}h(x)\, dx=
 \frac{2^{2m -\a}\pi^{n/2}\Gam (m\! -\!\a/2)}{\Gam ((n\!+\!\a\!-\!2m)/2)}=\gam_n (2m -\a);\nonumber\ee
 cf. the normalizing constant in (\ref{rpot}).
If $\a$ is not an integer, then $c_{\a,m}\!\neq \!0$ for any $m\!>\!\a/2$. If $\a \!\in \!\{1,2, \ldots, n\!-\!1\}$,
 then $c_{\a,m}\neq 0$ provided, e.g.,  that $m=[(\a+2)/2]$.
\end{lemma}
\begin {proof} The first statement can be easily checked by differentiation. For instance, we can write $\Del$ in polar coordinates
to get
$$
w(x)\!=\!(L^m \psi)(|x|^2), \quad L\!=\!4 r^{1-n/2}\, \frac{d}{dr}\, r^{n/2}\,  \frac{d}{dr},   \quad \psi(r)\!=\!(1\!+\!r)^{m-(n+\a)/2}.$$
This gives $w(x)=O ((1+|x|^2)^{-(\a+n)/2})\in  L^1(\bbr^n)$.
To prove the second statement we invoke the Bessel-McDonald kernel
\be G_\a (\xi)=\lam_\a \,\frac{K_{(n-\a)/2} (|\xi|)}{|\xi|^{(n-\a)/2}}, \qquad \lam_\a=\frac{2^{1-(\a+n)/2}}{\pi^{n/2}\Gam (\a/2)},\ee
where $K_\nu (\cdot)$ denotes the modified Bessel function (or the McDonald function) of order $\nu$  with the property \be \label{90k} K_{\nu} (z)= K_{-\nu} (z);\ee see, e.g., \cite {Wat}. It is known that
\be\label{hxi}
G_\a(\xi)\le c_\a\,
\begin{cases} |\xi|^{(\a -n-1)/2}e^{-|\xi|} & \text{if $|\xi|>1$,}
\\
|\xi|^{\a -n} &\text{if $\;|\xi|<1, \; \a<n$,}
\\
1  &\text{if $\; |\xi|<1, \; \a>n$,}
\\
1+\log (1/|\xi|) &\text{if $\; |\xi|<1, \; \a=n$,}
\end{cases}
\ee
where $ c_\a$ is a continuous function of $\a$; see, e.g., \cite [p. 257]{Ru96},  \cite [p. 285]{Nik}. Furthermore, for any
$\a>0$ the Fourier transform of $(1+|x|^2)^{-\a/2}$ in the sense of distributions  is computed by the formula
\be \label {56d}((1+|x|^2)^{-\a/2}, \phi (x))=(2\pi)^{-n} (G_\a(\xi), \hat \phi (\xi)),\qquad \phi \in S (\bbr^n),\ee
 \cite [Section 8.1]{Nik}.
Let us prove that
\be \label {0cbc} (I^\a w,  \phi)=a_{\a,m}\, ((1+|x|^2)^{(\a-2m-n)/2}, \phi (x))\ee
 for any test function $\phi$ in the  space $\Phi (\bbr^n)$. Once this is done, the required pointwise equality  will follow owing to Proposition \ref{c3.9}. We have
 \bea (I^\a w,  \phi)&=&(w,  I^\a\phi)=((1+|x|^2)^{(2m-n-\a)/2}, (-\Del)^m (I^\a\phi)(x))\nonumber\\
 &=&(2\pi)^{-n} (G_{n+\a-2m}(\xi), |\xi|^{2m-\a}\hat \phi (\xi))\nonumber\\
 &=&(2\pi)^{-n} \lam_{n+\a-2m}\, \left ( \frac{K_{m-\a/2} (|\xi|)}{|\xi|^{m-\a/2}},|\xi|^{2m-\a}\hat \phi (\xi)\right).\nonumber\eea
 Using (\ref{90k}) and (\ref {56d}), we  continue:
\bea(I^\a w,  \phi)&=&\frac{(2\pi)^{-n} \lam_{n+\a-2m}}{\lam_{2m+n-\a}}\, (G_{2m+n-\a}(\xi),\hat \phi (\xi))\nonumber\\
 &=&a_{\a,m}\, ((1+|x|^2)^{(\a-2m-n)/2}, \phi (x)).\nonumber\eea
 Evaluation of $c_{\a,m}$ is straightforward. To complete the proof, we note that  $c_{\a,m}\neq 0$ if $(n+\a-2m)/2\neq 0, -1, -2, \ldots \,$. The latter is guaranteed  under the afore-mentioned  choice of $m$.
\end{proof}

\noindent {\bf Proof of Theorem A.} The first statement follows immediately from Lemmas \ref{00jy}  and \ref{761s}. To prove the second  statement, we first note that $\tilde w= -\Del w$ belongs to $L^1$. Hence,  the convolution $I^\a f  \ast t^{-\a} \tilde w_t$ is well-defined in the Lebesgue sense and can be written as $f \ast \tilde h_t$ with $\tilde h=I^\a \tilde w=I^\a (-\Del) w$. Furthermore, for any test function $\phi \in \Phi (\bbr^n)$, (\ref{0cbc}) yields
\bea
(\tilde h,  \phi)&=&(I^\a (-\Del) w,  \phi)=-(I^\a w, \Del\,  \phi)\nonumber\\&=&-a_{\a,m}\, (\Del\,[(1+|x|^2)^{(\a-n)/2 - m}], \phi).\nonumber\eea
Since $\Del\,[(1+|x|^2)^{(\a-n)/2 - m}]=O((1+|x|^2)^{(\a-n)/2 - m-1})\in L^1$, then, by Proposition \ref{c3.9}, we have a pointwise equality
$$\tilde h (x)=-a_{\a,m}\, \Del\,[(1+|x|^2)^{(\a-n)/2 - m}].$$  Denote
$$
J_\e= \int_\e^\infty \frac{I^\a f  \ast  \tilde w_t}{t^{1+\a}}\, dt,\qquad \e>0.
$$
Then
$$
J_\e=\int_\e^\infty \frac{f\ast \tilde h_t}{t}\, dt=f \ast \psi_\e,  \qquad \psi_\e (x)=\int_\e^\infty \frac{ \tilde h_t (x)}{t}\, dt.
$$
Setting $\tilde h (x)=\tilde h_0 (|x|^2)$, we get
$$ \psi_\e (x)=\int_\e^\infty  \frac{\tilde h_0 (|x|^2/t^2)}{t^{n+1}}\, dt=\frac{1}{2\,|x|^n}\, \int_0^{x|^2/\e^2} \tilde h_0 (r)\, r^{n/2 -1}\, dr.
$$
This is a scaled version of the function $\psi (x)\!=\!2^{-1}|x|^{-n}\int_0^{|x|^2} \tilde h_0 (r)\, r^{n/2 -1}\, dr$. Observing that $$ \tilde h_0 (r)=-a_{\a,m}\, L\,[(1+r)^{(\a-n)/2 - m}],\qquad  L\!=\!4 r^{1-n/2}\, \frac{d}{dr}\, r^{n/2}\,  \frac{d}{dr},$$ we have
\bea
 \psi (x)&=&-\frac{2\,a_{\a,m}}{|x|^n}\, \int_0^{x|^2} \frac{d}{dr}\, r^{n/2}\,  \frac{d}{dr}\,[(1+r)^{(\a-n)/2 - m}]\, dr\nonumber \\
 &=&-\frac{2\,a_{\a,m}}{|x|^n}\,\left \{  r^{n/2}\,  \frac{d}{dr}\,[(1+r)^{(\a-n)/2 - m}]\right \}_{r=|x|^2}\nonumber \\
 &=&-a_{\a,m}\,(\a-n-2m)\,(1+|x|^2)^{(\a-n)/2 - m-1}.\nonumber \eea
This gives $\lim\limits_{\e \to 0} \,J_\e=d_{\a,m} f(x),$ where
$$
d_{\a,m}=\!\int_{\bbr^n}\!\!\psi (x)\,dx= \frac{2^{2m -\a+1}\pi^{n/2}\Gam (m\!+\!1\! -\!\a/2)}{\Gam ((n\!+\!\a\!-\!2m)/2)}=(2m\!-\!\a)\,\gamma_n(2m\!-\!\a),
$$
and the limit is understood in the required sense.

\section{Inversion of the Radon-John Transform}\label{333}

We recall  basic definitions. More information can be found in
\cite{ GGG, He, Ru04b}. Let $\ \mathcal{G}_{n,k}$ \ and $
G_{n,k}$ be the affine Grassmann manifold of all non-oriented
$k$-dimensional planes $\tau $ \ in $\mathbb{R}^{n}$
and the
ordinary Grassmann manifold of $k$-dimensional linear subspaces $\zeta $ of $%
\mathbb{R}^{n}$, respectively. Each $k$-plane $\tau \in
\mathcal{G}_{n,k}$ is parameterized as $\tau =\left( \zeta \text{,
}u\right) $, where $\zeta \in G_{n,k}$ and $u\in \zeta ^{\perp }$
(the orthogonal complement of $\zeta $ in $\mathbb{R}^{n}$).  The \textit{\ k-plane Radon-John
transform } of a function $f$ on $\mathbb{R}^{n}$ is defined by
\begin{equation}
(R_k f) (\tau )\equiv (R_k f) (\zeta \text{, }%
u)=\int_{\zeta }f(y+u)\,dy,  \label{3.1}
\end{equation}
where $dy$ is the induced Lebesque measure on the subspace \ $\zeta
\in G_{n,k}.$ \  This transform assigns to a function $f$ a
collection of integrals of $f$  over all $k$-planes in $\rn$. The
corresponding dual  transform of a function
$\varphi$ on $ \mathcal{G}_{n,k}$ is defined as the mean value of
$\varphi \left( \tau \right) $ over all $k$-planes $\tau $ through
$x\in \mathbb{R}^{n}$:
\be
(R_k^* \varphi)(x) =\int_{O(n)}\varphi (\mathcal{\sigma }\zeta
_{0}+x) \, d\mathcal{\sigma }, \qquad x\in \mathbb{R}^{n}.  \label{3.2}
\ee
Here $\zeta _{0}\in G_{n,k}$ \ is an arbitrary fixed $k-$plane
through the origin. If $f\in L_{p}(\mathbb{R}^{n}),$ then $R_k f$
is finite a.e. on $\mathcal{G}_{n,k}$ \ if and only if $1\leq
p<n/k$ \cite{So, Ru04b}.

A variety of  inversion procedures are known for $R_k f$; see \cite{ GGG, He, Ru04a, Ru04b}. One of the most important algorithms  relies on  the Fuglede
formula \cite{F},
\begin{equation}
R_k^* R_k  f=d_{k,n}I^{k}f, \qquad
d_{k,n}=\left( 2\pi \right) ^{k}\sigma _{n-k-1}/\sigma _{n-1},
\label{3.3}
\end{equation}
with the Riesz potential $I^{k}f$ on the right-hand side. Hence, Theorem A yields
 the following result. We denote
$$w(x)\!= \!(-\Del)^m [(1\!+\!|x|^2)^{m-(n+k)/2}], \quad m\!=\![(k\!+\!2)/2], \quad \tilde w(x)\!=\! (-\Del w)(x);$$
$$
 \lam_{k}=\frac{4^{m}\pi^{(n+k)/2}\Gam (n/2)\, \Gam (m -k/2)}{\Gam ((n -k)/2)\,\Gam ((n\!+\!k)/2-\!m)}, \qquad \del_{k}=
 (2m\!-\!k)\,\lam_{k}.
 $$
\noindent {\bf Theorem B.} {\it
 If $\varphi = R_k  f$, $f\in L^{p}(\bbr^n)$, $1\leq p<n/k$, then
\be \label {8rfhr} \lam_{k}\, f=\lim\limits_{t \to 0} \,(R_k^* \vp \ast t^{-k} w_t), \qquad  \del_{k}\, f=\int_0^\infty \frac{R_k^* \vp \ast  \tilde w_t}{t^{1+k}}\, dt, \ee
 where $\int_0^\infty (\ldots) =\lim\limits_{\e \to 0}\int_\e^\infty (\ldots)$.
The limit in both formulae exists in the $L^p$-norm and in the a.e. sense. If, moreover, $f\in C_0(\bbr^n)$, then it exists in the $\sup$-norm.}


\end{document}